\documentclass{article}
\pdfoutput=1
\usepackage{amsthm}
\usepackage{graphicx}

\author{
Jean-Guillaume Dumas\footnote{
Universit\'e de Grenoble. 
Laboratoire LJK,
umr CNRS, INRIA, UJF, UPMF, GINP.
51, av. des Math\'ematiques, F38041 Grenoble, France.
}~\footnotemark[3], 
Cl\'ement Pernet\footnote{
Universit\'e de Grenoble.
Laboratoire LIG,
umr CNRS, INRIA, UJF, UPMF, GINP.
51, av. J. Kuntzmann, F38330 Montbonnot St-Martin, France.
}~\footnotemark[3]
and Ziad Sultan\footnotemark[1]~\footnotemark[2]
\footnote{
\href{mailto:Jean-Guillaume.Dumas@imag.fr}{Jean-Guillaume.Dumas@imag.fr},
\href{mailto:Clement.Pernet@imag.fr}{Clement.Pernet@imag.fr},
\href{mailto:Ziad.Sultan@imag.fr}{Ziad.Sultan@imag.fr}.
}
}

\newcommand{\strechparskip}[1]{}
\newcommand{\strechparsep}[1]{}
\newcommand{\customvspace}[1]{}

\newcommand{\breakalgorithm}{}
\newcommand{\breakalgorithmsinglecolumn}{\State\Comment{the trailing
    parts of the algorithm are shown on next pages}\breakalgorithm{}}

\newfont{\seaddfnt}{phvr8t at 8pt}

\usepackage{mdwlist}
\usepackage[utf8]{inputenc}
\usepackage{xspace}
\usepackage{amsmath,amssymb}
\usepackage{booktabs}
\usepackage{algorithm}
\usepackage{algpseudocode}
\algrenewcommand\algorithmicrequire{\textbf{Input:}}
\algrenewcommand\algorithmicensure{\textbf{Output:}}
\algrenewcommand\algorithmicreturn{\textbf{Return}}
\algrenewcommand\Return{\State\algorithmicreturn{} }%

\usepackage{multirow}
\newcommand{\pluq}{\texttt{PLUQ}\xspace}
\newcommand{\ple}{\texttt{PLE}\xspace}
\newcommand{\cupd}{\texttt{CUP}\xspace}

\newcommand{\trsm}{\texttt{TRSM}\xspace}
\newcommand{\MM}{\texttt{MM}\xspace}
\newcommand{\mm}{\frac{m}{2}}
\newcommand{\nn}{\frac{n}{2}}
\newcommand{\Z}{\ensuremath{\mathbb{Z}}\xspace}
\newcommand{\GO}[1]{\ensuremath{O\left(#1\right)}\xspace}
\newcommand{\po}[1]{\ensuremath{o\left(#1\right)}\xspace}
\newtheorem{theorem}{Theorem}
\newtheorem{lemma}{Lemma}

\newtheorem{corrolary}{Corrolary}
\newtheorem{remark}{Remark}

\renewcommand{\breakalgorithm}{\algstore{bkbreak}\end{algorithmic}\end{algorithm}
\begin{algorithm}[htbp]
\begin{algorithmic}\algrestore{bkbreak}}

\title{Simultaneous computation of the row and column rank profiles}

\makeatletter
\usepackage{color,svgcolor}
\usepackage{hyperref}
\hypersetup{
pdftitle={\@title},
pdfauthor={Jean-Guillaume Dumas, Clément Pernet and Ziad Sultan},
breaklinks=true, 
colorlinks=true,
 linkcolor=darkred,
 citecolor=blue,
 urlcolor=darkgreen,
}
\makeatother

\begin{document}
\maketitle

\abstract{
Gaussian elimination with full pivoting generates a PLUQ matrix decomposition. 
Depending on the strategy used in the search for pivots, the
permutation matrices can reveal some information about the row or the column
rank profiles of the matrix. We propose a new pivoting strategy that makes it
possible to recover at the same time both row and column rank profiles of the input matrix and of any  of its leading sub-matrices.
We propose a rank-sensitive and quad-recursive algorithm that computes
the latter PLUQ triangular decomposition of an $m\times n$ matrix of rank $r$ in
\GO{mnr^{\omega-2}} field operations, with $\omega$ the exponent of matrix
multiplication. 
Compared to the LEU decomposition by Malashonock, sharing a similar recursive
structure, its time complexity is rank sensitive and has a lower leading
constant.
Over a word size finite field, this algorithm also improveLs the practical
efficiency of previously known implementations.
}
\section{Introduction}
Triangular matrix decomposition is a fundamental building block in computational
linear algebra. It is used to solve linear systems, compute the rank,
the determinant, the nullspace or the row and column rank profiles of a matrix.
The LU decomposition, defined for matrices whose leading principal
minors are all nonsingular, can be generalized to arbitrary dimensions and ranks
by introducing pivoting on sides, leading e.g. to the LQUP decomposition
of~\cite{Ibarra:1982:LSP} or the PLUQ decomposition~\cite{GoVa96,Jeffrey:2010:lufact}.
Many algorithmic variants exist allowing fraction free computations
\cite{Jeffrey:2010:lufact}, in-place computations~\cite{jgd:2008:toms,JPS:2011}
or sub-cubic rank-sensitive time complexity~\cite{Storjohann:2000:thesis,JPS:2011}. 
More precisely, the pivoting strategy reflected by the permutation matrices
$P$ and $Q$ is the key difference between these PLUQ decompositions. 
In numerical linear algebra~\cite{GoVa96}, pivoting
is used to ensure a good numerical stability, good data locality, and reduce
the fill-in. In the context of exact linear algebra, the role of pivoting
differs. 
Indeed, only certain pivoting strategies for these decompositions will reveal
the rank profile of the matrix. The latter  is crucial
in many applications using exact Gaussian elimination, such as Gr\"obner basis
computations~\cite{F99a} and computational number
theory~\cite{stein2007modular}. 

The {\em row rank profile} of an $m\times n$ matrix with rank $r$ is a
lexicographically smallest sequence of $r$ row indices such that the
corresponding rows of the matrix are linearly independent. 
Similarly the {\em column rank profile} is a lexicographically smallest sequence
of $r$ column indices such that the corresponding rows of the matrix are
linearly independent.

The common strategy to compute the row rank profile is to search for pivots in
a row-major fashion: exploring the current row, then moving to the next row only
if the current row is zero. Such a PLUQ decomposition can be transformed into a
CUP decomposition (where $C=PL$ is in column echelon form) and the first
$r$ values of the permutation associated to $P$ are exactly the row rank
profile~\cite{JPS:2011}. A block recursive algorithm can be derived from this scheme by
splitting the row dimension~\cite{Ibarra:1982:LSP}.
Similarly, the column rank profile can be obtained in a column major search:
exploring the current column, and moving to the next column only if the
current one is zero. The PLUQ decomposition can be transformed into a PLE
decomposition (where $E=UQ$ is in row echelon form) and the first $r$ values of
$Q$ are exactly the column rank profile~\cite{JPS:2011}. The corresponding block
recursive algorithm uses a splitting of the column dimension.

Recursive elimination algorithms splitting both row and column dimensions
include the TURBO algorithm~\cite{jgd:2002:PComp} and the LEU
decomposition~\cite{Malaschonok:2010}. 
No connection is made to the computation of the rank profiles in any of them.
The TURBO algorithm does not compute the lower triangular matrix $L$ and
performs five recursive calls. It therefore implies an arithmetic overhead
compared to classic Gaussian elimination.
The LEU decomposition aims at reducing the amount of permutations and therefore also
uses many additional matrix products. As a consequence its time complexity
is not rank-sensitive.


We propose here a pivoting strategy following a Z-curve structure and
working on an incrementally growing leading sub-matrix. 
This strategy is first used in a recursive algorithm splitting both rows
and columns which recovers simultaneously both row and column rank profiles.
Moreover, the row and column rank profiles of any
leading sub-matrix can be deduced from the $P$ and $Q$ permutations.
We show that the arithmetic cost of this algorithm remains rank sensitive of the
form $O(mnr^{\omega-2})$ where $\omega$ is the exponent of matrix
multiplication. The best currently known upper bound for $\omega$ is
$2.3727$~\cite{Williams:2012:matmul}.
As for the CUP and PLE decompositions, this PLUQ decomposition can be computed
in-place. We also propose an iterative variant, to be used as a base-case.  
 

Compared to the CUP and PLE decompositions, this new algorithm has the following
new salient features: 
\customvspace{-1em}
 \begin{itemize}
\strechparskip{-3pt}
\strechparsep{-6pt}
 \item it computes {\em simultaneously} both rank profiles at the cost of one,
 \item it preserves the squareness of the matrix passed to the recursive calls,
  thus allowing more efficient use of the matrix multiplication building block,
 \item it reduces the number of modular reductions in a finite field,
\item a CUP and a PLE decompositions can be obtained from it, with row and
  column permutations only.
 \end{itemize} 
\customvspace{-1em}
Compared to the LEU decomposition, 
\customvspace{-1em}
\begin{itemize}
\strechparskip{-3pt}
\strechparsep{-6pt}
\item it is in-place,
\item its time complexity bound is rank sensitive and has a better leading constant,
\item a LEU decomposition can be obtained from it, with row and column
  permutations.
\end{itemize}
\customvspace{-1em}

In Section~\ref{sec:rec} we present the new block recursive algorithm. 
Section~\ref{sec:leu} shows the connection with the LEU
decomposition and section~\ref{sec:rp} states the main property about
rank profiles.
We then analyze the complexity of the new algorithm in terms
of arithmetic operations: first we prove that it is rank
sensitive in Section~\ref{sec:comp} and second we show in
section~\ref{sec:mod} that, over a finite field, it reduces the number of
modular reductions when compared to state of the art techniques.
We then propose an iterative variant in Section~\ref{sec:base} to be
used as a base-case to terminate the recursion before the dimensions get too
small. Experiments comparing computation time and cache efficiency are presented
in section~\ref{sec:exp}.

\section{A recursive PLUQ algorithm}\label{sec:rec}

\setlength{\arraycolsep}{.8\arraycolsep}
We first  recall the name of the main sub-routines being used: \MM
stands for matrix multiplication, \trsm for triangular system solving
with matrix unknown (left and right variants are implicitly indicated
by the parameter list), \texttt{PermC} for matrix column permutation,
\texttt{PermR} for matrix row permutation, etc. For instance, we will use:
\customvspace{-12pt}
\begin{description}
\strechparskip{0pt}
\strechparsep{-1pt}
\item[$\MM(C,A,B)$] to denote $C\leftarrow C-AB$,
\item[$\trsm(U,B)$] for $B\leftarrow
  U^{-1}B$ with $U$ upper triangular,
\item[$\trsm(B,L)$] for $B\leftarrow BL^{-1}$ with $L$ lower triangular.
\end{description}
\customvspace{-12pt}
We  also denote by $T_{k,l}$ the transposition of indices $k$ and $l$ and by 
$L\backslash U$, the storage of the two triangular matrices $L$ and $U$ one
above the other. 
Further details on these subroutines and notations can be found in~\cite{JPS:2011}.
In block decompositions, we  allow for zero dimensions. By convention,
the product of any $m\times 0$ matrix by an $0\times n$ matrix is  the
$m\times n$ zero matrix.

We now present the block recursive algorithm~\ref{alg:pluq:4rec},
computing a PLUQ decomposition. 

{\scriptsize
\begin{algorithm}[H]
  \caption{\pluq}
  \label{alg:pluq:4rec}
\begin{algorithmic}
\Require{$A=(a_{ij})$ a $m\times n$ matrix over a field}
\Ensure{$P, Q$:  $m\times m$ and $n\times n$ permutation matrices}
\Ensure{$r$: the rank of $A$}
\Ensure{$A \leftarrow
\begin{bmatrix}
    L \backslash U & V\\
    M              &0
  \end{bmatrix}$ where 
$L$ is $r\times r$ unit lower triangular, 
$U$ is $r\times r$ upper triangular,
and $$A= P \begin{bmatrix} L\\M \end{bmatrix} \begin{bmatrix} U&V \end{bmatrix} Q.$$
}
\If{m=1}
   \If{$A=\begin{bmatrix} 0& \ldots& 0\end{bmatrix}$}
      $P\leftarrow  \begin{bmatrix}  1 \end{bmatrix}, Q\leftarrow I_n,
      r\leftarrow 0$
   \Else
      \State $i\leftarrow$ the col. index of the first non zero elt. of $A$
      \State $P\leftarrow  \begin{bmatrix}  1 \end{bmatrix};Q\leftarrow T_{1,i},
      r\leftarrow 1$
      \State Swap $a_{1,i}$ and $a_{1,1}$
   \EndIf
   \Return $(P, Q, r, A)$
\EndIf
\If{n=1}
   \If{$A=\begin{bmatrix} 0& \ldots& 0\end{bmatrix}^T$}
       $P\leftarrow I_m ; Q\leftarrow  \begin{bmatrix}  1 \end{bmatrix}, 
      r\leftarrow 0$
   \Else
      \State $i\leftarrow$ the row index of the first non zero elt. of $A$
      \State $P\leftarrow  \begin{bmatrix}  1 \end{bmatrix}, Q\leftarrow T_{1,i},
      r\leftarrow 1$
      \State Swap $a_{i,1}$ and $a_{1,1}$
      \For{$j=i+1 \dots m$}
         $a_{j,1}\leftarrow a_{j,1} a_{1,1}^{-1}$
      \EndFor
   \EndIf
   \Return $(P, Q, r, A)$
\EndIf

\breakalgorithmsinglecolumn{}

\State \Comment{Splitting $A=
    \begin{bmatrix}
      A_{1}&A_{2}\\
      A_{3}&A_{4}\\
    \end{bmatrix}$ where $A_{1}$ is $\lfloor\frac{m}{2}\rfloor\times \lfloor\frac{n}{2}\rfloor$.}
  \State Decompose $A_{1} = P_1
  \begin{bmatrix}L_1\\M_1\end{bmatrix}\begin{bmatrix}U_1&V_1\end{bmatrix}
  Q_1$ \Comment{$\pluq(A_{1})$}
  \State $\begin{bmatrix} B_1\\B_2\end{bmatrix}
\leftarrow P_1^TA_{2}$ \Comment{$\texttt{PermR}(A_2,P_1^T)$}
  \State $
  \begin{bmatrix}
    C_1&C_2
  \end{bmatrix}
 \leftarrow A_{3}Q_1^T$ \Comment{$\texttt{PermC}(A_3,Q_1^T)$}
  \State Here $A =
    \left[\begin{array}{cc|c}
      L_1 \backslash U_1& V_1& B_1\\
      M_1               & 0  & B_2\\
      \hline
      C_1               & C_2& A_{4}\\
    \end{array}\right]$.
   \State $D\leftarrow L_1^{-1}B_1$ \Comment{$\trsm(L_1,B_1)$}
   \State $E\leftarrow C_1U_1^{-1}$ \Comment{$\trsm(C_1,U_1)$}
   \State $F\leftarrow B_2-M_1D$ \Comment{$\MM(B_2,M_1,D)$}
   \State $G\leftarrow C_2-EV_1$ \Comment{$\MM(C_2,E,V_1)$}
   \State $H\leftarrow A_4-ED$ \Comment{$\MM(A_4,E,D)$}
   \State Here $A=
     \left[\begin{array}{cc|c}
       L_1 \backslash U_1& V_1& D\\
       M_1               & 0  & F\\
       \hline
       E               & G& H\\
     \end{array}\right]$.
   
   \State Decompose $F =
   P_2\begin{bmatrix}L_2\\M_2\end{bmatrix}\begin{bmatrix}U_2&V_2\end{bmatrix}
   Q_2$ \Comment{$\pluq(F)$}
   \State Decompose $G =
   P_3\begin{bmatrix}L_3\\M_3\end{bmatrix}\begin{bmatrix}U_3&V_3\end{bmatrix}
   Q_3$ \Comment{$\pluq(G)$}
   \State $ \begin{bmatrix} H_1&H_2\\H_3&H_4 \end{bmatrix}
\leftarrow P_3^THQ_2^T$ \Comment{$\texttt{PermR}(H,P_3^T);\texttt{PermC}(H,Q_2^T)$}
   \State $\begin{bmatrix} E_1\\E_2 \end{bmatrix} \leftarrow  P_3^T E$ \Comment{$\texttt{PermR}(E,P_3^T)$}
   \State $\begin{bmatrix} M_{11}\\M_{12} \end{bmatrix} \leftarrow  P_2^T M_1$ \Comment{$\texttt{PermR}(M_1,P_2^T)$}
   \State $\begin{bmatrix} D_1&D_2\end{bmatrix} \leftarrow  DQ_2^T$ \Comment{$\texttt{PermR}(D,Q_2^T)$}
   \State $\begin{bmatrix} V_{11}&V_{12}\end{bmatrix} \leftarrow  V_1Q_3^T$ \Comment{$\texttt{PermR}(V_1,Q_3^T)$}

\breakalgorithm{}

\State Here $A=
    \left[\begin{array}{ccc|ccc}
      L_1 \backslash U_1& V_{11}&V_{12}& D_1 &D_2\\
      M_{11}               & 0  &0& L_2\backslash U_2 & V_2\\
      M_{12}               & 0  &0& M_2 & 0\\
      \hline
      E_1               & L_3\backslash U_3&V_3& H_1 &H_2\\
      E_2               & M_3              &0  & H_3 &H_4\\
    \end{array}\right]$.
   \State $I\leftarrow H_1U_2^{-1}$ \Comment{$\trsm(H_1,U_2)$}
   \State $J\leftarrow L_3^{-1}I$ \Comment{$\trsm(L_3,I)$}
    \State $K\leftarrow H_3U_2^{-1}$ \Comment{$\trsm(H_3,U_2)$}
    \State $N\leftarrow L_3^{-1}H_2$ \Comment{$\trsm(L_3,H_2)$}
    \State    $O\leftarrow N-JV_2$ \Comment{$\MM(N,J,V_2)$}
    \State $R\leftarrow H_4-KV_2-M_3O$ \Comment{$\MM(H_4,K,V_2);\MM(H_4,M_3,O)$}
    \State  Decompose $ R =
    P_4\begin{bmatrix}L_4\\M_4\end{bmatrix}\begin{bmatrix}U_4&V_4\end{bmatrix}
    Q_4$ \Comment{$\pluq(R)$}
    \State $ \begin{bmatrix}
      E_{21}& M_{31} & 0 & K_1\\
      E_{22}& M_{32} & 0 & K_2\\
    \end{bmatrix} 
    \leftarrow  P_4^T  \begin{bmatrix} E_{2}& M_{3} & 0 & K\\ \end{bmatrix}$ \Comment{$\texttt{PermR}$}
    \State $\begin{bmatrix} D_{21}&D_{22}\\ V_{21}&V_{22}\\ 0&0 \\ O_1&O_2 \end{bmatrix}
            \leftarrow\begin{bmatrix} D_{2}\\ V_{2}\\ 0 \\ O\end{bmatrix} Q_4^T$ \Comment{$\texttt{PermC}$}

    \State Here $A=
    \left[\begin{array}{ccc|cccc}
      L_1 \backslash U_1& V_{11}&V_{12}& D_1 &D_{21}&D_{22}\\
      M_{11}               & 0  &0& L_2\backslash U_2 & V_{21}&V_{22}\\
      M_{12}               & 0  &0& M_2 & 0&0\\
      \hline
      E_1               & L_3\backslash U_3&V_3& I &O_1 &O_2\\
      E_{21}               & M_{31}              &0  & K_1 & L_4\backslash U_4 &V_4\\
      E_{22}               & M_{32}              &0  & K_2 & M_4 &0\\
    \end{array}\right]$.

\State $S\leftarrow \begin{bmatrix}
  I_{r_1+r_2}\\
  &&I_{k-r_1-r_2}\\
  &I_{r_3+r_4}\\
  &&&&I_{m-k-r_3-r_4}
\end{bmatrix}$
\State $T\leftarrow \begin{bmatrix}
  I_{r_1}\\
        &      &          &I_{r_2} & \\
        &I_{r_3}&          &       & \\
        &      &          &       & I_{r_4}\\
        &      &I_{k-r_1-r_3}\\
        &      &          &       &       & I_{n-k-r_2-r_4}\\
\end{bmatrix}$

 \State $P\leftarrow \text{Diag}(
   P_1
   \begin{bmatrix}
     I_{r_1}\\&P_2
   \end{bmatrix},P_3
   \begin{bmatrix}
     I_{r_3}\\&P_4
   \end{bmatrix} ) S$
 \State $Q\leftarrow
T\text{Diag}(\begin{bmatrix}
     I_{r_1}\\&Q_3
   \end{bmatrix}
 Q_1,\begin{bmatrix}
     I_{r_2}\\&Q_4
   \end{bmatrix}
 Q_2)$

 
\State $A\leftarrow S^TAT^T$ \Comment{$\texttt{PermR}(A,S^T); \texttt{PermC}(A,T^T)$}
\State Here $A=
\begin{bmatrix}
  L_1\backslash U_1 &D_1              &V_{11} &D_{21} & V_{12}&D_{22}\\
    M_{11}          &L_2\backslash U_2 & 0 &       V_{21} & 0 & V_{22}\\ 
     E_1           &       I         &L_3\backslash U_3 & O_1 & V_3&O_2 \\
     E_{21}         &K_1              & M_{31}            & L_4\backslash U_4     &0&V_4 \\
     M_{12}         & M_2             & 0 &0 &0& 0\\
     E_{22}         & K_2             & M_{32}           & M_4 & 0 &0\\
\end{bmatrix}$
\Return $(P,Q,r_1+r_2+r_3+r_4,A)$
\end{algorithmic}
\end{algorithm}
}

It is based on a splitting of the matrix in four quadrants. A first recursive
call is done on the upper left quadrant followed by a series of updates. Then
two recursive calls can be made on the anti-diagonal quadrants if the first
quadrant exposed some rank deficiency.
After a last series of updates, a fourth recursive call is done on the bottom
right quadrant.
Figure~\ref{fig:pluq:rec} illustrates the position of the blocks computed in the
course of algorithm~\ref{alg:pluq:4rec}, before and after the final permutation
with matrices $S$~and~$T$.
\begin{figure}[ht]\centering
\includegraphics[width=\linewidth]{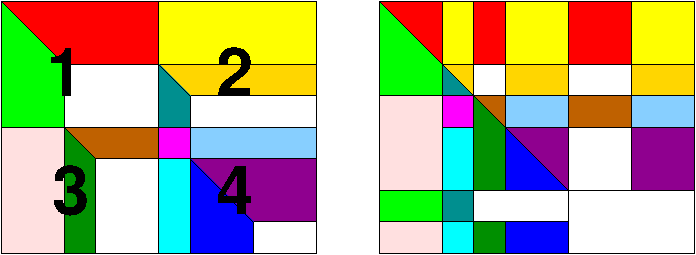}
\caption{Block recursive Z-curve PLUQ decomposition and final block permutation.}
\label{fig:pluq:rec}
\customvspace{-3pt}
\end{figure}

This framework differs from the one in~\cite{jgd:2002:PComp} by the order in
which the quadrants are treated, leading to only four recursive calls in this
case instead of five in~\cite{jgd:2002:PComp}. We will show in
section~\ref{sec:rp} that this fact together with the special form of the block
permutations $S$ and $T$ makes it possible to recover rank profile information. 
The correctness of algorithm~\ref{alg:pluq:4rec} is proven in appendix~\ref{app:pluq:correct}.
\begin{remark}
Algorithm~\ref{alg:pluq:4rec} is in-place (as defined
in~\cite[Definition 1]{JPS:2011}):
all operations of the \trsm, \MM,
\texttt{PermC}, \texttt{PermR} subroutines work with $O(1)$ extra memory
allocations except possibly in the course of fast matrix multiplications.
The only constraint is for the computation of $J\leftarrow L_3^{-1}I$ which
would overwrite the matrix $I$ that should be kept  for the final output. 
Hence a copy of $I$ has to  be stored for the computation of $J$. The matrix $I$
has dimension $r_3\times r_2$ and can be stored transposed in the zero block of
the upper left quadrant (of dimension $(\mm-r_1)\times (\nn-r_1)$, as
shown on Figure~\ref{fig:pluq:rec}).
\end{remark}
%

\section{From PLUQ to LEU}\label{sec:leu}

\renewcommand{\arraystretch}{0.91}
\setlength{\arraycolsep}{.12\arraycolsep}

We now show how to compute the LEU decomposition of~\cite{Malaschonok:2010}
from the PLUQ decomposition. The idea is to write 
$$P
\begin{bmatrix} L\\M \end{bmatrix}
\begin{bmatrix}U&V\end{bmatrix}Q= 
\underbrace{
P
\begin{bmatrix}
  L&0\\
  M&I_{m-r}
\end{bmatrix}
P^T}_{\overline{L}} 
\underbrace{P
\begin{bmatrix}
  I_r\\&0
\end{bmatrix}
Q}_{E}
\underbrace{
 Q^T
\begin{bmatrix}
  U&V\\
  & I_{n-r}
\end{bmatrix}
Q}_{\overline{U}}$$
 and show that $\overline{L}$ and $\overline{U}$ are respectively lower and
 upper triangular. This is not true in general, but turns out to be
 satisfied by the $P,L,M,U,V$ and $Q$ obtained in
 algorithm~\ref{alg:pluq:4rec}.%
\setlength{\arraycolsep}{.13\arraycolsep}
\begin{theorem}
Let $A=P
\begin{bmatrix}
  L\\M
\end{bmatrix}
\begin{bmatrix}
  U&V
\end{bmatrix}
Q$ be the \pluq decomposition computed by algorithm~\ref{alg:pluq:4rec}.
Then for any unit lower triangular matrix $Y$ and any upper triangular matrix $Z$,
the matrix
$  P
  \begin{bmatrix}
    L \\
    M&Y
  \end{bmatrix}
P^T$ is  unit lower triangular and
$Q^T
\begin{bmatrix}
  U&V\\&Z
\end{bmatrix}Q$
is upper triangular.
\end{theorem}

\begin{proof}
Proceeding by induction, we assume that the theorem is true on all four
recursive calls, and show that it is true for the matrices 
$P
  \left[\begin{smallmatrix}
    L \\
    M&Y
  \end{smallmatrix}\right]
P^T$
and 
$Q^T
  \left[\begin{smallmatrix}
    U&V \\
    &Z
  \end{smallmatrix}\right]
Q$.
Let $Y=
\begin{bmatrix}
  Y_1\\Y_2&Y_3
\end{bmatrix}$
where $Y_{1}$ is unit lower triangular of dimension $k-r_1-r_2$.
From~the correctness of algorithm~\ref{alg:pluq:4rec} (see e.g. Equation~\ref{eq:PL}),
$  S
  \begin{bmatrix}
    L\\
    M&Y
  \end{bmatrix}
  S^T
  =
      \left[
        \begin{array}{ccccccc}
          L_1  \\
          M_{11} & L_2\\
          M_{12} & M_2 &Y_1&       \\
          \hline
          E_1   &  I  &&L_3\\
          E_{21} &K_1 & &M_{31}  & L_4\\
          E_{22} &K_2 & Y_2&M_{32}  & M_4&Y_3\\
        \end{array}
      \right]
 $
%

Hence $ P \left[\begin{matrix}L \\M&Y\end{matrix}\right] P^T $ equals
  \begin{equation*}
    \begin{split}
      \begin{bmatrix}P_1\\&P_3 \end{bmatrix}
      \begin{bmatrix} I_{r_1}\\&P_2\\&&I_{r_3}\\&&&P_4 \end{bmatrix}
      \left[
        \begin{array}{ccccccc}
          L_1  \\
          M_{11} & L_2\\
          M_{12} & M_2 &Y_1&       \\
          \hline
          E_1   &  I  &&L_3\\
          E_{21} &K_1 & &M_{31}  & L_4\\
          E_{22} &K_2 & Y_2&M_{32}  & M_4&Y_3\\
        \end{array}
      \right] \times
\\
    \begin{bmatrix}  I_{r_1}\\&P_2^T\\&&I_{r_3}\\&&&P_4^T \end{bmatrix}
    \begin{bmatrix}  P_1^T\\&P_3^T\end{bmatrix}
   \end{split}
   \end{equation*}
By induction hypothesis, the matrices
$\overline{L_2}=P_2 \begin{bmatrix}L_2\\M_2&Y_1\end{bmatrix}P_2^T,$
$\overline{L_4}=P_4 \begin{bmatrix}L_4\\M_4&Y_3\end{bmatrix}P_4^T$
,
$P_1
\begin{bmatrix}
  L_1\\
  M_1&\overline{L_2}
\end{bmatrix}
P_1^T
$
and
$P_3
\begin{bmatrix}
  L_3\\
  M_3&\overline{L_4}
\end{bmatrix}
P_3^T
$ are unit lower triangular. Therefore the matrix  $P
\left[\begin{smallmatrix}
  L\\
  M&Y
\end{smallmatrix}\right]
P^T$ is also unit lower triangular.

Similarly, let $Z=
\begin{bmatrix}
  Z_1&Z_2\\&Z_3
\end{bmatrix}$
where $Z_1$ is upper triangular of dimension $k-r_1-r_2$.
The matrix $ T^T \begin{bmatrix}U&V \\&Z\end{bmatrix}T$ equals
\begin{eqnarray*}
  T^T \left[
        \begin{array}{ccc|cccc}
          U_1  & V_{11} & V_{12} &D_1 &D_{21}&D_{22}\\
                &0     & 0     &U_2 &V_{21}&V_{22}\\
                &U_3   & V_3    & 0   &O_1 & O_2\\
                &      & 0     &    & U_4& V_4\\
                &      & Z_1    &    &   & Z_2\\
                &      &        &    &   &Z_3\\
        \end{array}
      \right]
  &=& \left[
        \begin{array}{ccc|cccc}
          U_1  & V_{11} & V_{12} &D_1 &D_{21}&D_{22}\\
                &U_3   & V_3    &    &O_1 & O_2\\
                &      & Z_1    &    &   & Z_2\\
                &0     & 0     &U_2 &V_{21}&V_{22}\\
                &      & 0     &    & U_4& V_4\\
                &      &        &    &   &Z_3\\
        \end{array}
      \right]
\end{eqnarray*}
Hence $ Q^T\left[\begin{matrix}U&V \\&Z\end{matrix}\right]Q $ equals
  \begin{equation*}
    \begin{split}
 \begin{bmatrix} Q_1^T\\&Q_2^T\end{bmatrix}
    \begin{bmatrix} I_{r_1}\\&Q_3^T \\&&I_{r_2}\\&&&Q_4^T \end{bmatrix}
    \left[
        \begin{array}{ccc|cccc}
          U_1  & V_{11} & V_{12} &D_1 &D_{21}&D_{22}\\
                &U_3   & V_3    &    &O_1 & O_2\\
                &      & Z_1    &    &   & Z_2\\
                &0     & 0     &U_2 &V_{21}&V_{22}\\
                &      & 0     &    & U_4& V_4\\
                &      &        &    &   &Z_3\\
        \end{array}
      \right]\times \\
      \begin{bmatrix} I_{r_1}\\&Q_3\\&& I_{r_2}\\&&&Q_4 \end{bmatrix}
   \begin{bmatrix}Q_1\\&Q_2 \end{bmatrix}.
\end{split}
\end{equation*}
By induction hypothesis, the matrices
$\overline{U_3}=Q_3^T \begin{bmatrix}U_3&V_3\\&Z_1\end{bmatrix}Q_3$,
$\overline{U_4}=Q_4^T \begin{bmatrix}U_4&V_4\\&Z_3\end{bmatrix}P_4^T$
, $Q_1^T \begin{bmatrix}  U_1&V_1\\     &\overline{U_3} \end{bmatrix} Q_1$
and
$Q_2^T\begin{bmatrix}  U_2&V_2\\  &\overline{U_4}\end{bmatrix}Q_2$ are upper triangular.
Consequently the matrix  $Q^T\left[\begin{smallmatrix}  U&V\\  &Z
\end{smallmatrix}\right]Q$ is upper triangular.

For the base case with $m=1$. The matrix $\overline{L}$ has dimension $1\times
1$ and is unit lower triangular. If $r=0$, then $\overline{U}=I_n^T Z I_n$ is
upper triangular. If $r=1$, then $Q=T_{1,i}$ where $i$ is the column index of
the pivot and is therefore the column index of the  leading coefficient of the row
$\begin{bmatrix}U&V\end{bmatrix}Q$. Applying $Q^T$ on the left only swaps rows 1
  and $i$, hence row $\begin{bmatrix}U&V\end{bmatrix}Q$ is the $i$th row of
    $Q^T\begin{bmatrix}  U&V\\&Z\end{bmatrix}Q$. The latter is
therefore upper triangular. The same reasoning can be applied to the case $n=1$.
%
\end{proof}
\begin{corrolary}\label{cor:leu}
  Let $\overline{L} = P
  \begin{bmatrix}
    L\\M&I_{m-r}
  \end{bmatrix}P^T, 
E=P
\begin{bmatrix}
  I_r\\&0
\end{bmatrix}
Q$ and $\overline{U} = Q^T
  \begin{bmatrix}
    U&V\\&0
  \end{bmatrix}Q
  $.
Then $A=\overline{L}E\overline{U}$ is a LEU decomposition of $A$.
\end{corrolary}


\begin{remark}
The converse is not always possible: given $A=L,E,U$, there are
several ways to choose the last $m-r$ columns of $P$ and the last $n-r$ rows of
$Q$. The LEU algorithm does not keep track of these parts of
the permutations.
\end{remark}


\section{Computing the rank profiles}
\label{sec:rp}

We prove here the main feature of the PLUQ decomposition computed by
algorithm~\ref{alg:pluq:4rec}: it reveals the row and column rank profiles of all leading
sub-matrices of the input matrix.
%
We recall in Lemma~\ref{lem:rankprofiletransf} basic properties verified by the
rank profiles.
\begin{lemma} For any matrix,
\label{lem:rankprofiletransf}
\begin{enumerate}\customvspace{-5pt}\strechparsep{-1pt}\strechparskip{0pt}
\item\label{point:rrp} the row rank profile is preserved by right multiplication
  with an invertible matrix and by left multiplication with an invertible upper
  triangular matrix.
\item\label{point:crp} the column rank profile is preserved by left multiplication with an
  invertible matrix and by right multiplication with an invertible lower
  triangular matrix.
\customvspace{-5pt}\end{enumerate}
\end{lemma}

  

\begin{lemma}
  Let $A=PLUQ$ be the \pluq decomposition computed by algorithm~\ref{alg:pluq:4rec}.
  Then the row (resp. column) rank profile of any leading $(k,t)$ submatrix of
  $A$ is the row (resp. column) rank profile of the leading $(k,t)$ submatrix
  of $P
  \begin{bmatrix}
    I_r\\&0
  \end{bmatrix}Q
  $.
\end{lemma}
\begin{proof}With the notations of corollary~\ref{cor:leu}, we have:
  \begin{equation*}
    A= P 
  \begin{bmatrix}
    L\\M&I_{m-r}
  \end{bmatrix}
  \begin{bmatrix}
    I_r\\&0
  \end{bmatrix}
  \begin{bmatrix}
    U&V\\&I_{n-r}
  \end{bmatrix}
Q
= \overline{L} P
\begin{bmatrix}
  I_r\\&0
\end{bmatrix}
Q\overline{U}
 \end{equation*}
Hence $$
    \begin{bmatrix}
      I_k&0
    \end{bmatrix}
    A
    \begin{bmatrix}
      I_t\\0
    \end{bmatrix}
= \overline{L_1}    
\begin{bmatrix}
  I_k&0
\end{bmatrix}
P
\begin{bmatrix}
  I_r\\&0
\end{bmatrix}
Q\overline{U_1},
$$
where $\overline{L_1}$ is the $k\times k$ leading submatrix of $\overline{L}$
(hence it is an invertible lower triangular matrix) and $\overline{U_1}$ is the
  $t\times t$ leading submatrix of $\overline{U}$ (hence it is an invertible
  upper triangular matrix). 
Now, Lemma~\ref{lem:rankprofiletransf} implies  that the rank profile of 
$
\begin{bmatrix}
  I_k&0
\end{bmatrix}A
\begin{bmatrix}
  I_t\\0
\end{bmatrix}$
is that of 
$
\begin{bmatrix}
  I_k&0
\end{bmatrix}
P
\begin{bmatrix}
  I_r\\&0
\end{bmatrix}
Q
\begin{bmatrix}
  I_t\\0
\end{bmatrix}$.
\end{proof}


From this lemma we deduce how to compute the row and column rank profiles of any
$(k,t)$ leading submatrix and more particularly of the matrix $A$ itself.

\begin{corrolary}
  Let $A=PLUQ$ be the \pluq decomposition of a $m\times n$ matrix computed by  algorithm~\ref{alg:pluq:4rec}.
  The row (resp. column) rank profile of any $(k,t)$-leading submatrix of a
   $A$ is the sorted sequence of the row
  (resp. column) indices of the non zero rows (resp. columns) in the matrix 
$$  R=\begin{bmatrix} I_k& 0 \end{bmatrix}
  P
  \begin{bmatrix} I_r\\&0\end{bmatrix}
  Q
  \begin{bmatrix} I_t\\0 \end{bmatrix}
  $$
\end{corrolary}

\begin{corrolary}
  The row (resp. column) rank profile of $A$ is the sorted sequence of row
  (resp. column) indices of the non zero rows (resp. columns) of the first $r$
  columns of $P$ (resp. first $r$ rows of $Q$).
\end{corrolary}

\section{Complexity analysis}\label{sec:comp}

We study here the time complexity of algorihtm~\ref{alg:pluq:4rec} by counting
the number of field operations.
For the sake of simplicity, we will assume here that the dimensions $m$ and $n$
are powers of two. The analysis can easily be extended to the general case for
arbitrary $m$ and $n$.

For $i=1,2, 3, 4$ we denote by $T_i$ the cost of the $i$-th recursive call to
\pluq, on a $\frac{m}{2}\times\frac{n}{2}$ matrix of rank $r_i$.
We also denote by $T_\trsm(m,n)$ the cost of a call \trsm on a rectangular
matrix of dimensions $m\times n$, and by $T_\MM(m,k,n)$ the cost of multiplying
an $m\times k$ by an $k\times n$ matrix.

\begin{theorem}
  Algorithm~\ref{alg:pluq:4rec}, run on an $m\times n$ matrix of rank
  $r$, performs $\GO{mnr^{\omega-2}}$ field operations.
\end{theorem}

\begin{proof}
Let $T=T_\pluq(m,n,r)$ be the cost of algorithm~\ref{alg:pluq:4rec} run on a
$m\times n$ matrix of rank $r$. From the complexities of the subroutines given, e.g., in~\cite{jgd:2008:toms} and the recursive calls in algorithm~\ref{alg:pluq:4rec}, we have:
\begin{eqnarray*}
  T &=& T_1+T_2+T_3+T_4
 + T_{\trsm}(r_1, \mm) + T_{\trsm}(r_1, \nn)\\
&&+ T_{\trsm}(r_2, \mm) + T_{\trsm}(r_3, \nn) +T_\MM(\mm-r_1,r_1,\nn)\\
&& +T_\MM(\mm,r_1,\nn-r_1) +T_\MM(\mm,r_1,\nn)\\
&&+ T_\MM(r_3,r_2,\nn-r_2)+ T_\MM(\mm-r_3,r_2,\nn-r_2-r_4)\\
&& +T_\MM(\mm-r_3,r_3,\nn-r_2-r_4)\\
&\leq& T_1+T_2+T_3+T_4
+ K\left(\mm(r_1^{\omega-1}+r_2^{\omega-1})+\nn(r_1^{\omega-1}\right.\\
&&\left.+ r_3^{\omega-1}) + \mm \nn r_1^{\omega-2} + \mm \nn
  r_2^{\omega-2} + \mm \nn r_3^{\omega-2}\right)\\ 
&\leq&  T_1+T_2+T_3+T_4
  + K'mnr^{\omega-2} 
\end{eqnarray*}

for some constants $K$ and $K'$ (we recall that $a^{\omega-2}+b^{\omega-2} \leq
2^{3-\omega}(a+b)^{\omega-2}$ for $2\leq \omega\leq 3$).

Let $C=max\{\frac{K'}{1-2^{4-2\omega}};1\}$.
Then we can prove by a simultaneous induction on $m$ and $n$ that 
$ T \leq Cmnr^{\omega-2}$.

Indeed, if $(r=1,m=1,n\geq m)$ or $(r=1,n=1,m\geq n)$ then $T\leq m-1\leq Cmnr^{\omega-2}$.
Now if it is true for $m=2^j,n=2^i$, then for $m=2^{j+1},n=2^{i+1}$, we have
 \begin{eqnarray*}
 T &\leq& \frac{C}{4}mn(r_1^{\omega-2}+r_2^{\omega-2}+r_3^{\omega-2}+r_4^{\omega-2}) +K'mnr^{\omega-2} \\
   &\leq& \frac{C(2^{3-\omega})^2}{4}mnr^{\omega-2}   +K'mnr^{\omega-2} \\
   &\leq& K'\frac{2^{4-2\omega}}{1-2^{4-2\omega}}mnr^{\omega-2}   +K'mnr^{\omega-2} \leq Cmnr^{\omega-2}.
 \end{eqnarray*}
\end{proof}

In order to compare this algorithm with usual Gaussian elimination algorithms,
we now refine the analysis to compare the leading constant of the time
complexity in the special case where the matrix is square and has a generic rank
profile: $r_1=\mm=\nn, r_2=0,  r_3=0$ and $r_4=\mm=\nn$ at each recursive step. 

Hence we have
\begin{eqnarray*}
  T_\pluq &=& 2T_\pluq(\nn,\nn,\nn) + 2T_\trsm(\nn,\nn)+ T_\MM(\nn,\nn,\nn)\\
         &=&  2T_\pluq(\nn,\nn,\nn) +
  2\frac{C_\omega}{2^{\omega-1}-2}\left(\nn\right)^\omega + C_\omega\left(\nn\right)^\omega
\end{eqnarray*}

Writing $T_\pluq(n,n,n) = \alpha n^\omega$, the constant $\alpha$ satisfies:   
$$
\alpha=C_\omega\frac{1}{(2^{\omega}-2)}\left(\frac{1}{2^{\omega-2}-1}+1\right) = C_\omega\frac{2^{\omega-2}}{(2^{\omega}-2)(2^{\omega-2}-1)}.
$$
which is equal to the constant of the CUP and LUP decompositions~\cite[Table
  1]{JPS:2011}. In particular, 
it equals $2/3$ when $\omega=3, C_\omega=2$, matching the constant of the classical
Gaussian elimination.


\section{Number of modular reductions over a prime field}
\label{sec:mod}

In the following we suppose that the operations are done with full
delayed reduction for a single multiplication and any number of
additions: operations of the form $\sum a_i b_i$ are reduced only once
at the end of the addition, but $a \cdot b \cdot c$ requires two
reductions.
In practice, only a limited amount of accumulations can be done on an
actual mantissa without overflowing, but we neglect this in this
section for the sake of simplicity. 
See e.g. \cite{jgd:2008:toms} for more details. 
For instance, with this model, the number of reductions required by a
classic multiplication of matrices of size $m\times k$ by $k\times n$
is simply: $m\cdot n$. We denote this by $R_{MM}(m,k,n)=mn$.
This extends e.g. also for triangular solving:

\begin{theorem}\label{thm:trsm}
Over a prime field modulo $p$, the number of reductions modulo $p$ required by
$TRSM(m,n)$ with full delayed reduction is:
$$
\begin{array}{lcll}
R_\text{UnitTRSM}(m,n)&=&mn& \text{ if the triangular matrix is
  unitary,}\\
R_\text{TRSM}(m,n)&=&2mn& \text{ in general.}
\end{array}
$$
\customvspace{-5pt}
\end{theorem}
\begin{proof} 
If the matrix is unitary, then a fully delayed reduction is required
only once after the update of each row of the result.
In the generic case, we invert each diagonal element first and
multiply each element of the right hand side by this inverse diagonal
element, prior to the update of each row of the result. This gives
$mn$ extra reductions. 
\end{proof}

Next we show that the new pivoting strategy is more efficient in terms
of number of integer division. 
\begin{theorem} Over a prime field modulo $p$ and on a full-rank
  square $m\times m$ matrix with generic rank profile, and $m$ a power
  of two, the number of reductions modulo $p$ required by the
 elimination algorithms with full delayed reduction is:
\customvspace{-5pt}
$$  \begin{array}{lcl}
  R_{\pluq}(m,m) &=&2m^2+\po{m^2},\\
  R_{\ple}(m,m)=R_{\cupd}(m,m) &=&\left(1+\frac{1}{4}\log_2(m)\right)m^2+\po{m^2}
  \end{array}
$$
\customvspace{-10pt}
\end{theorem}
\begin{proof}
 If the top left square block is full rank then \pluq reduces to
  one recursive call, two square TRSM (one unitary, one generic) one
  square matrix multiplication and a final recursive call. In terms of modular reductions, this gives: 
$R_{\pluq}(m)=2R_{\pluq}(\frac{m}{2})+R_{\texttt{UnitTRSM}}(\frac{m}{2},\frac{m}{2})+R_{\trsm}(\frac{m}{2},\frac{m}{2})+R_{\MM}(\frac{m}{2},\frac{m}{2},\frac{m}{2})$.
Therefore, using theorem~\ref{thm:trsm}, the number of reductions
within \pluq satisfies $T(m)=2T(\frac{m}{2})+m^2$ so that
it is $R_{PLUQ}(m,m) =2m^2-2m$ if $m$ is a power of two.

 For row or column oriented elimination this situation is more
  complicated since the recursive calls will always be
  rectangular even if the intermediate matrices are full-rank.
We in fact prove, by induction on $m$, the more generic:
\begin{equation}\label{eq:rcup}
R_{\ple}(m,n)=\log_2(m)(\frac{mn}{2}-\frac{m^2}{4})
            +m^2+\po{mn+m^2}
\end{equation}
First $R_{\ple}(1,n)=0$ since $[1]\times[a_1,\ldots,a_n]$ is a triangular decomposition of the $1\times n$ matrix $[a_1,\ldots,a_n]$.
Now suppose that Equation~\ref{eq:rcup} holds for $k = m$.
Then we follow the row oriented algorithm of \cite[Lemma 5.1]{jgd:2008:toms} which makes two recursive calls, one \trsm and one \MM to get  $R_{\ple}(2m,n)=R_{\ple}(m,n)+R_{\ple}(m,m)+R_{\MM}(m,m,n-m)+R_{\ple}(m,n-m)=R_{\ple}(m,n)+R_{\ple}(m,n-m)+m(n+m)$. 
We then apply the induction hypothesis on the recursive calls to get
\begin{eqnarray*}
R_{\ple}(2m,n)&=&\frac{1}{2}\log_2(m)mn-\frac{1}{4}\log_2(m)m^2+m^2+\\
&&\frac{1}{2}\log_2(m)m(n-m)-\frac{1}{4}\log_2(m)m^2+m^2+\\
&&m(n+m)+\po{mn+m^2}\\
& =& \log_2(m)(mn-m^2)+3m^2+mn+\po{mn+m^2}.
\end{eqnarray*}
The latter is also obtained by substituting $k\hookleftarrow 2m$ in Equation~\ref{eq:rcup} so that the induction is proven.
\customvspace{-15pt}\end{proof}

This show that the new algorithm requires much less modular reductions, as soon
as $m$ is larger than $32$. Over finite fields, since reductions can be much
more expensive than multiplications or additions by elements of the field, this
is a non negligible advantage. We show in the next section that this
participates to the better practical performance of the \pluq algorithm.

\section{A base case  algorithm}\label{sec:base}

We propose in algorithm~\ref{alg:pluq:iter} an iterative algorithm computing the same PLUQ decomposition as
algorithm~\ref{alg:pluq:4rec}. 
The motivation is to offer an alternative to the recursive algorithm 
improving the computational efficiency on small matrix sizes. 
Indeed, as long as the matrix fits the cache memory, the amount of
page faults of the two variants are similar, but the iterative
algorithm reduces the amount of row and column permutations.
The block recursive algorithm can then be modified so that it switches to the
iterative algorithm whenever the matrix dimensions are below a certain
threshold.

Unlike the common Gaussian elimination, where pivots are searched in the whole
current row or column, the strategy is here to proceed with an incrementally growing
leading sub-matrix. This implies a Z-curve type search scheme, as shown on
figure~\ref{fig:iter}.
This search strategy is meant to ensure the properties on the rank profile that
have been presented in section~\ref{sec:rp}.
\begin{figure}[h]
\begin{center}
  \includegraphics[width=.52\linewidth]{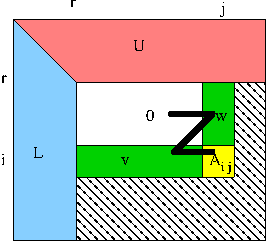}
\end{center}
  \caption{Iterative base case PLUQ decomposition}
\label{fig:iter}  
\end{figure}

{\scriptsize
\begin{algorithm}[H]
  \caption{\pluq iterative base case}
  \label{alg:pluq:iter}
\begin{algorithmic}
\Require{$A$ a $m\times n$ matrix over a field}
\Ensure{$P,Q$:  $m\times m$ and $n\times n$ permutation matrices}
\Ensure{$r$: the rank of $A$}
\Ensure{$A \leftarrow
\begin{bmatrix}
    L \backslash U & V\\
    M              &0
  \end{bmatrix}$ where $L$ is $r\times r$ unit lower triang., $U$ is $r\times
  r$ upper triang. and such that $A=
  P
  \begin{bmatrix}
    L\\M
  \end{bmatrix}
  \begin{bmatrix}
    U&V
  \end{bmatrix}
  Q
  $.
}

\State $r \leftarrow 0; i \leftarrow 0; j \leftarrow 0$
\While{$i<m$ or $j<n$}
\State \Comment{Let $v = \begin{bmatrix} A_{i,r}&\ldots&A_{i,j-1} \end{bmatrix}$ and
      $w = \begin{bmatrix} A_{r,j}&\ldots&A_{i-1,r} \end{bmatrix}^T$}
    \If{$j<n$ and $w \neq 0$}
        \State $p\leftarrow$ row index of the first non zero entry in $w$
        \State $q\leftarrow j; j\leftarrow \max(j+1,n)$
    \ElsIf{$i<m$ and $v \neq 0$}
        \State $q\leftarrow $ column index of the first non zero entry in $v$
        \State $p\leftarrow i; i\leftarrow \max(i+1,m)$
    \ElsIf{$i<m$ and $j<n$ and $A_{i,j} \neq 0$}
    \State $(p,q)\leftarrow (i,j)$
    \State $i\leftarrow \max(i+1,m); j\leftarrow \max(j+1,n)$
    \Else
    \State $i\leftarrow \max(i+1,m); j\leftarrow \max(j+1,n)$
    \State continue
    \EndIf
    \Comment{At this stage, $A_{p,q}$ is a pivot}
    \For{$k=p+1\ldots n$}
      \State $A_{k,q} \leftarrow A_{k,p} / A_{p,q}$
      \State $A_{k,q+1\ldots n} \leftarrow A_{k,q+1\ldots n} - A_{k,q}
      A_{p,q+1\ldots n}$
    \EndFor
    \State $A_{r+1\ldots m, r+1} \leftrightarrow A_{r+1\ldots m, q}$
    \Comment{Swap pivot column}
    \State $A_{r+1,r+1\ldots n} \leftrightarrow A_{p,r+1\ldots n}$ \Comment{Swap
      pivot row}
    \State $P\leftarrow T_{p,r} P ; Q\leftarrow Q T_{q,r}$ \Comment{$T_{k,l}$
      swaps indices $k$ and $l$}
    \State $r\leftarrow r+1$
\EndWhile
\end{algorithmic}
\end{algorithm}
}

\begin{remark}
In order to further improve the data locality, this iterative
algorithm can be transformed into a left-looking
variant~\cite{DonDufSorVor98}. We did not write this version here for
the sake of clarity, but this is how we implemented the base case for
the experiments of section~\ref{sec:exp}.
\end{remark}

\section{Experiments}\label{sec:exp}
We present here experiments comparing an implementation of
algorithm~\ref{alg:pluq:4rec} computing a PLUQ decomposition against the
implementation of the CUP/PLE decomposition, called \texttt{LUdivine} in the 
\texttt{FFLAS-FFPACK} library\footnote{\url{http://linalg.org/fflas-ffpack}}.
The new implementation of the PLUQ
decomposition is available in this same library from version svn@346.
We ran our tests on a single core of an Intel Xeon
E5-4620@2.20GHz using gcc-4.7.2.

Figures~\ref{fig:densefull} and~\ref{fig:densedeficient}  compare the computation time of
LUdivine, and the new PLUQ algorithm. In figure~\ref{fig:densefull}, the matrices are dense, with full rank. The
computation times are similar, the PLUQ algorithm with base case showing a
slight improvement over LUdivine. 
\begin{figure}[htb]
   \includegraphics[width=\linewidth]{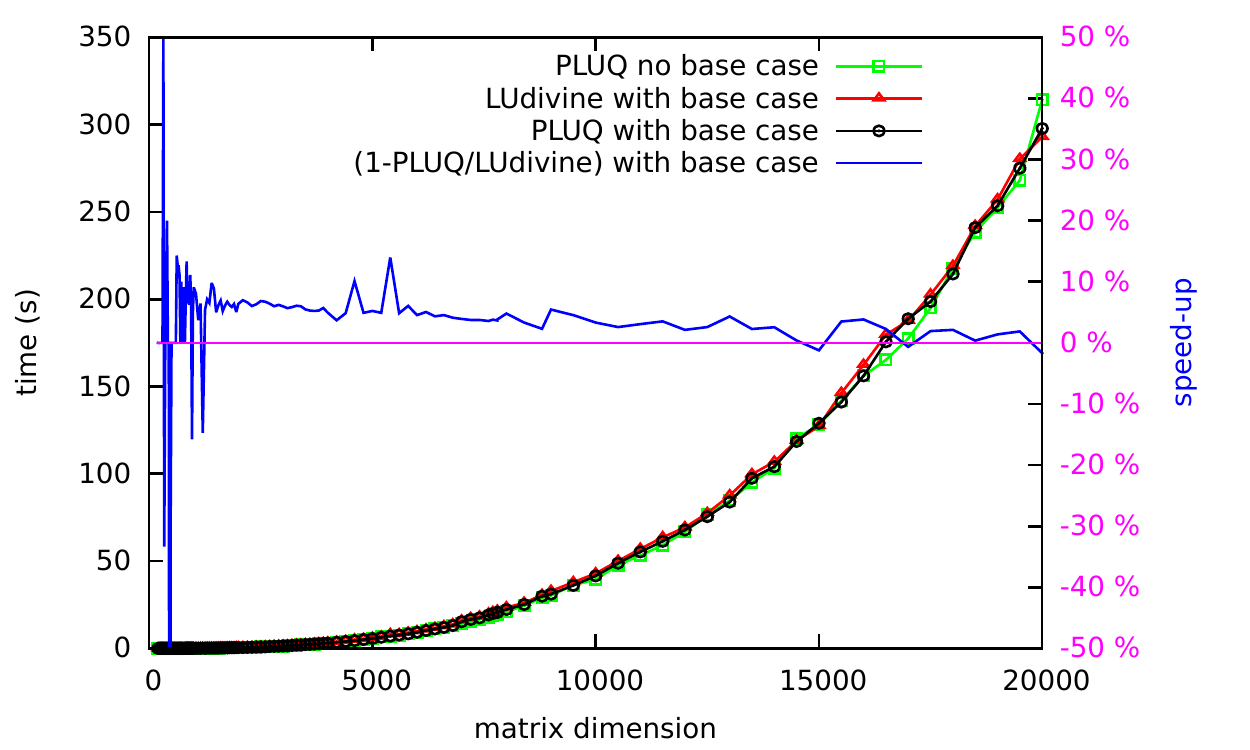}
   \caption{Computation time with dense full rank matrices over $\Z/1009\Z$.}
   \label{fig:densefull}
\end{figure}
\begin{figure}[thb]
   \includegraphics[width=\linewidth]{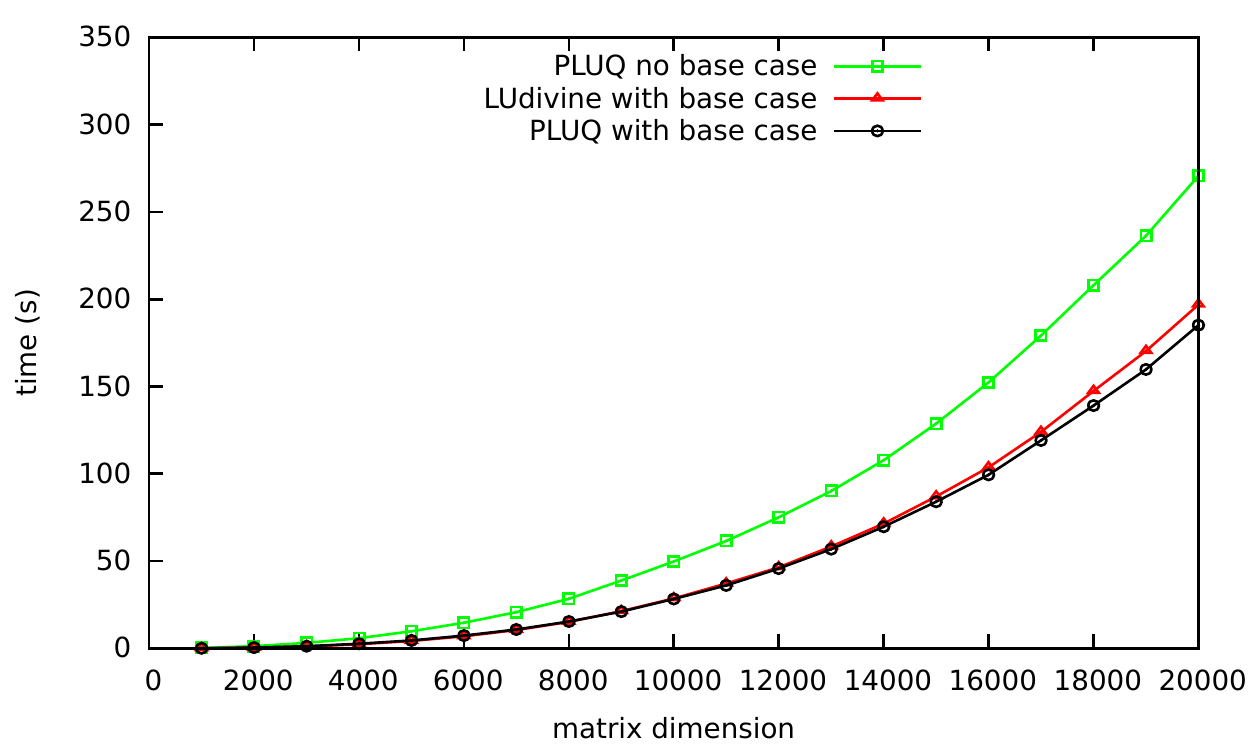}
   \caption{Computation time with dense rank deficient matrices (rank is half
     the dimension)}
   \label{fig:densedeficient}
\end{figure}
\begin{figure}[htb]
   \includegraphics[width=\linewidth]{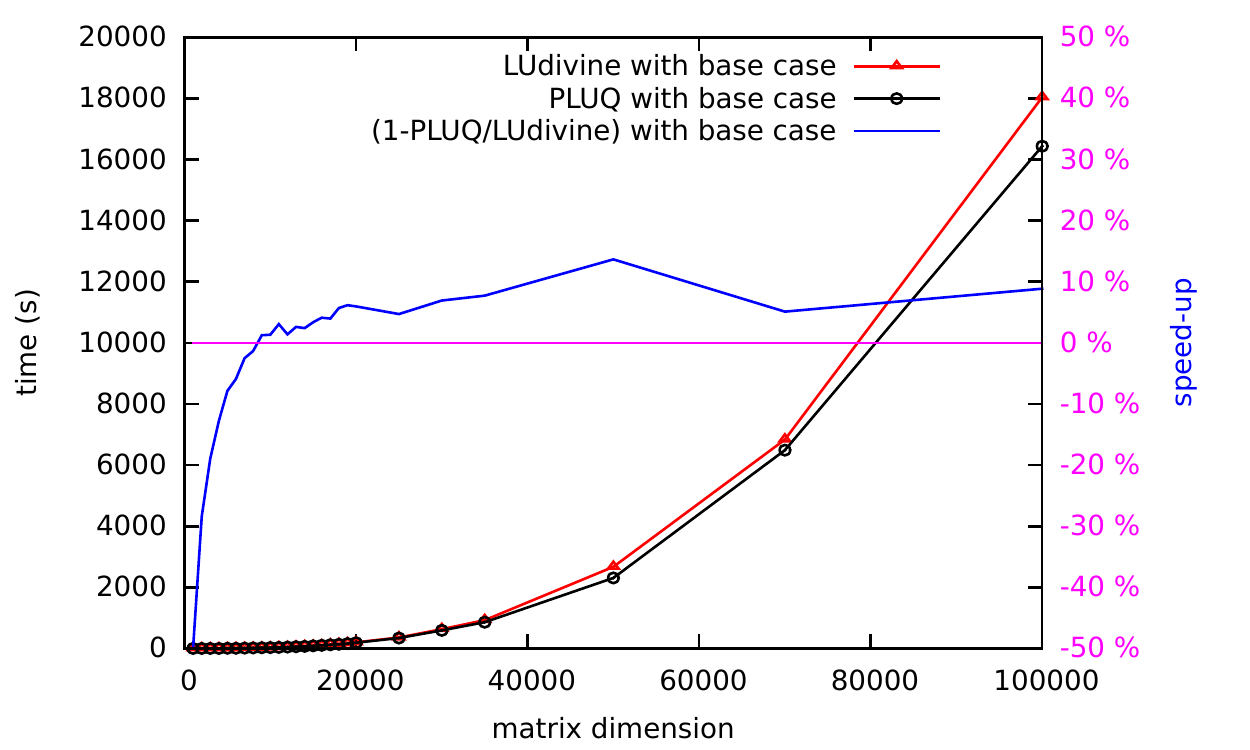}
   \caption{Computation time with dense rank deficient matrices of larger dimension}
   \label{fig:densedeficientlarge}
\end{figure}
In figures ~\ref{fig:densedeficient} and~\ref{fig:densedeficientlarge}, the
matrices are square, dense with a rank equal to half the dimension. To ensure
non trivial row and column rank profiles, they are generated from a LEU
decomposition, where $L$ and $U$ are uniformly random non-singular lower and
upper triangular matrices, and $E$ is zero except on $r=n/2$ positions, chosen
uniformly at random, set to one.
The cutoff dimension for the switch to the base case has been set to an optimal
value of $30$ by experiments. Figure~\ref{fig:densedeficient}
shows how the base case greatly improves the efficiency for PLUQ, presumably for
it reduces the number of row and column permutations. With the base case the computation
time is comparable to LUdivine. More precisely, PLUQ becomes faster than
LUDivine for dimensions above 9000. Figure~\ref{fig:densedeficientlarge} shows
that, on larger matrices, PLUQ can be about 10\% faster than LUdivine.



Table~\ref{tab:cachemisses} summarizes some of the data reported by the callgrind tool
of the valgrind emulator (version 3.8.1) concerning the cache misses. We also
report in the last column the corresponding computation time on the machine (without emulator). 
The matrices used are the same as in figure~\ref{fig:densedeficient}, with
rank half the dimension.
We first  notice the impact of the base case on the PLUQ algorithm: although it
does not change the number of cache misses, it strongly reduces the total number
of memory accesses (less permutations), thus improving the computation time.
Now as the dimension grows, the total amount of memory accesses and
the amount of cache misses plays in favor of PLUQ which becomes faster
than LUdivine. 
\begin{table*}[htbp]
\begin{center}
\begin{tabular}{llrrrrr}
\toprule
 Matrix & Algorithm & Accesses & L1 Misses & LL Misses & Relative  & Timing (s) \\
\midrule
\multirow{3}{*}{A4K} & PLUQ-no-base-case & 1.319E+10 & 7.411E+08 & 1.523E+07 & .115 & 5.84 \\
    & PLUQ-base-case & 8.119E+09 & 7.414E+08 & 1.526E+07 & .188 & 2.65 \\
    & LUdivine & 1.529E+10 & 1.246E+09 & 2.435E+07 & .159 & \textbf{2.35} \\
\midrule
\multirow{3}{*}{A8K} & PLUQ-no-base-case & 6.150E+10 & 5.679E+09 & 1.305E+08 & .212 & 28.4\\
& PLUQ-base-case & 4.072E+10 & 5.681E+09 & 1.306E+08 & .321 & 15.4 \\
 & LUdivine & 7.555E+10 & 9.693E+09 & 2.205E+08 & .292 & \textbf{15.2}\\
\midrule
\multirow{3}{*}{A12K} & PLUQ-no-base-case & 1.575E+11 & 1.911E+10 & 4.691E+08 & .298 & 75.1 \\
& PLUQ-base-case & 1.112E+11 & 1.911E+10 & 4.693E+08 & .422 & \textbf{45.7} \\
 & LUdivine & 2.003E+11 & 3.141E+10 & 7.943E+08 & .396 & 46.4\\
\midrule
\multirow{3}{*}{A16K} & PLUQ-no-base-case & 3.142E+11 & 4.459E+10 & 1.092E+09 & .347 & 152 \\
 & PLUQ-base-case & 2.302E+11 & 4.459E+10 & 1.092E+09 & .475 & \textbf{99.4} \\
 & LUdivine & 4.117E+11 & 7.391E+10 & 1.863E+09 & .452 & 103 \\
\bottomrule
\end{tabular}
\end{center}
\caption{Cache misses for dense matrices with rank equal half of the dimension}  
\label{tab:cachemisses} 
\end{table*}


\section{Conclusion and perspectives}

The decomposition that we propose can first be viewed as an improvement over the
LEU decomposition, introducing a finer treatment of rank deficiency that reduces
the number of arithmetic operations, makes the time complexity rank sensitive
and allows to perform the computation in-place.

Second, viewed as a variant of the existing CUP/PLE decompositions, this new
algorithm produces more information on the rank profile and has better cache
efficiency, as it avoids calling matrix products with rectangular matrices of
unbalanced dimensions. It also  performs fewer modular reductions when computing
over a finite field. 

Overall the new algorithm is also faster in practice than previous
implementations when matrix dimensions get large enough. 

Now, in a parallel setting, it should exhibit more parallelism than row or
column major eliminations since the recursive calls in step 2 and 3 are
independent. This is also the case for the
TURBO algorithm of~\cite{jgd:2002:PComp}, but the latter requires more
arithmetic operations. Further experiments and analysis of communication costs
have to be conducted in shared and distributed memory settings to assess the
possible practical gains in parallel.

{\small
\bibliographystyle{abbrvurl}
\bibliography{pluq}
}
\appendix
\renewcommand{\arraystretch}{.7}
\setlength{\arraycolsep}{.2\arraycolsep}

\section{Correctness of algorithm~1}
\label{app:pluq:correct}

First note that 
$S
\begin{bmatrix}
  L\\M
\end{bmatrix}
=\left[
\begin{array}{cccccc}
  L_1  \\
  M_{11} & L_2 \\
  M_{12} & M_2 &0\\
  \hline
E_1   &  I  &L_3\\
  E_{21} &K_1 & M_{31}  & L_4\\
  E_{22} & K_2& M_{32}  & M_4 & 0 &0\\
    \end{array}\right]
$

Hence
$  P
  \begin{bmatrix}
    L\\M
  \end{bmatrix}
=
  \begin{bmatrix}
    P_1\\&P_3
  \end{bmatrix}
\left[
  \begin{array}{cccccc}
  L_1  \\
  M_1 & P_2
  \begin{bmatrix}
    L_2\\M_2
  \end{bmatrix}\\
  \hline
  E_1   &  I  &L_3\\
  E_{2} &K & M_{3}  & P_4
  \begin{bmatrix}
    L_4\\M_4
  \end{bmatrix}&
  \end{array}
\right]
\label{eq:PL}
$

Similarly,
$
  \begin{bmatrix}
    U&V
  \end{bmatrix}
T =
\left[
  \begin{array}{ccc|ccc}
    U_1 & V_{11} & V_{12}& D_1 & D_{21}& D_{22}\\
        &0      &0     & U_2 & V_{21}& V_{22}\\
        &U_3    & V_3  &  0   & O_1   & O_2\\
        &       &      &    & U_4   & V_4\\
        &       &      &    &       &0\\
  \end{array}
\right]
$
and 
$
  \begin{bmatrix}
    U&V
  \end{bmatrix}
  Q =
\left[
\begin{array}{cc|cccc}
  U_1 & V_1 &D_1 & D_2\\
      & 0   & U_2& V_2 \\
      &\begin{bmatrix}U_3&V_3\end{bmatrix} Q_3 &0 & O\\
      &                 &    &\begin{bmatrix}U_4&V_4\end{bmatrix} Q_4\\
\end{array}
\right]
  \begin{bmatrix}
    Q_1\\&Q_2
  \end{bmatrix}.
$
%

Now as
$H_1=IU_2, H_2=IV_2 + L_3 O, H_3=KU_2$ and
$H_4=KV_2+M_3O+P_4
\begin{bmatrix}
  L_4\\M_4
\end{bmatrix}
\begin{bmatrix}
  U_4&V_4
\end{bmatrix}
Q_4 $
%
we have
\begin{eqnarray*}
    P
  \begin{bmatrix}
    L\\M
  \end{bmatrix}
  \begin{bmatrix}
    U&V
  \end{bmatrix}
Q &=&   \begin{bmatrix}
    P_1\\&P_3
  \end{bmatrix}
\left[
  \begin{array}{cccccc}
  L_1 &   \\
  M_1 & P_2
  \begin{bmatrix}
    L_2\\M_2
  \end{bmatrix}\\
  \hline
  E_1   &  I  &L_3\\
  E_{2} &K & M_{3}  & P_4
  \begin{bmatrix}
    L_4\\M_4
  \end{bmatrix}&
  \end{array}
\right]
\\&&
\left[
\begin{array}{cc|cccc}
  U_1 & V_1 &D_1 & D_2\\
      & 0   & U_2& V_2 \\
      &\begin{bmatrix}U_3&V_3\end{bmatrix} Q_3 &0 & O\\
      &                 &    &\begin{bmatrix}U_4&V_4\end{bmatrix} Q_4\\
\end{array}
\right]
  \begin{bmatrix}
    Q_1\\&Q_2
  \end{bmatrix}
\\&=&   
\begin{bmatrix}
    P_1\\&P_3
  \end{bmatrix}
\left[
  \begin{array}{cccccc}
  L_1 &   \\
  M_1 & P_2
  \begin{bmatrix}
    L_2\\M_2
  \end{bmatrix}\\
  \hline
    E_1 &    & I_{r_3}\\
    E_2 &    & &I_{m-k-r_3}\\
  \end{array}
\right]
\\&&
\left[
\begin{array}{cc|cccc}
  U_1 & V_1 &D_1 & D_2\\
      & 0   & U_2& V_2 \\
      &L_3\begin{bmatrix}U_3&V_3\end{bmatrix} Q_3 &H_1 & H_2\\
      &M_3\begin{bmatrix}U_3&V_3\end{bmatrix} Q_3 &H_3  &H_4\\
\end{array}
\right]
  \begin{bmatrix}
    Q_1\\&Q_2
  \end{bmatrix}\\
&=&\begin{bmatrix}
    P_1\\&I_{m-k}
  \end{bmatrix}
\left[
  \begin{array}{cccccc}
  L_1 &   \\
  M_1 & \\
  \hline
    E &  0  & I_{m-k}\\
  \end{array}
\right]
\left[
\begin{array}{cc|cccc}
  U_1 & V_1 & D\\
      & 0   & F \\
      &G &H
\end{array}
\right]\\&&
  \begin{bmatrix}
    Q_1\\&I_{n-k}
  \end{bmatrix}\\
&=&
\begin{bmatrix}
    P_1\\&I_{m-k}
  \end{bmatrix}
\left[
\begin{array}{cc|c}
  L_1U_1 & L_1V_1 & B_1\\
  M_1U_1    & M_1V_1   & B_2 \\
   C_1 &  C_2  & A_4 \\
\end{array}
\right]
  \begin{bmatrix}
    Q_1\\&I_{n-k}
  \end{bmatrix}\\
&=& A
\end{eqnarray*}

\end{document}